\documentclass[reqno,11pt]{amsart}
\usepackage{amsmath, amsfonts,amsthm,amssymb,amsbsy,upref,color,graphicx,amscd,hyperref,enumerate}

\usepackage[active]{srcltx}
\usepackage[latin1]{inputenc}
\usepackage{bbm}
\usepackage{pgf}
\usepackage{xfrac}
\def\ds{\displaystyle}
\def\eps{{\varepsilon}}
\def\N{\mathbb{N}}

\def\R{\mathbb{R}}

\def\HH{\mathcal{H}}

\newcommand{\be}{\begin{equation}}
\newcommand{\ee}{\end{equation}}

\newcommand{\de}{\partial}

\newcommand{\bC}{{\bf C}}

\newcommand{\sing}{{\rm Sing}}

\newcommand{\freq}{\mathcal{S}}

\theoremstyle{plain}
\newtheorem{teo}{Theorem}

\newtheorem{lemma}[teo]{Lemma}

\theoremstyle{remark}

\newtheorem{oss}[teo]{Remark}

\title[Almost everywhere  uniqueness of blow-up limits]{Almost everywhere uniqueness of blow-up limits for the lower dimensional obstacle problem}

\author{Maria Colombo, Luca Spolaor, Bozhidar Velichkov}

\address {Maria Colombo: \newline \indent
Institute for Theoretical Studies, ETH Z\"urich,
	\newline \indent
 Clausiusstrasse 47, CH-8092 Z\"urich, Switzerland
 	}
\email{maria.colombo@epfl.ch}

\address {Luca Spolaor: \newline \indent
UC San Diego, 
	\newline \indent
	9500 Gilman Drive,
	La Jolla, CA 92093-0112, USA}
\email{lspolaor@ucsd.edu}

\address {Bozhidar Velichkov: \newline \indent
	Dipartimento di Matematica e Applicazioni
	"Renato Caccioppoli" \newline \indent
	Universit\`a degli Studi di Napoli Federico II\newline \indent
	Via Cintia, Monte S. Angelo
	I-80126 Napoli, Italy}
\email{bozhidar.velichkov@unina.it}

\makeindex

\begin{document}



\begin{abstract}
We answer a question left open in \cite{fosp2} and \cite{fosp2-cor}, by proving that the blow-up of minimizers $u$ of the lower dimensional obstacle problem is unique at generic point of the free-boundary. 
\end{abstract}

\maketitle


\section{Introduction}

Let $s\in (0,1)$, let $B_1$ be the unit ball in $\R^d$, where $d\geq 2$, and let $B_1' := B_1 \cap \{ x_{d}=0 \}$.
For any point $x=(x_1,\dots,x_d)\in\R^d$ we denote by $x'$ the vector of the first $(d-1)$ coordinates, $x'=(x_1,\dots,x_{d-1})$.  We consider the  class of admissible functions
$$\mathcal A:=\big\{ u\in H^1(B_1, x_d^{1-2s} \mathcal L^d)\,:\,u\geq 0 \mbox{ on }B_1'   \,,\,\,u(x',x_d)=u(x',-x_d) \mbox{ for every }(x',x_d)\in B_1 \big\}\,,$$
We say that $u\in\mathcal A$ is a \emph{solution of the lower dimensional obstacle problem} if 
\begin{equation}
\label{pbm:thin}
\int_{B_1}x_d^{1-2s} |\nabla u|^2\,dx\le \int_{B_1}x_d^{1-2s}|\nabla v|^2\,dx\quad\text{for every}\quad v\in\mathcal A\quad\text{such that}\quad u-v\in H^1_0(B_1).
\end{equation}

\noindent  For a solution $u\in \mathcal A$ of the lower dimensional obstacle problem, we define the \emph{coincidence set} $\Delta(u)$ as 
$$\Delta(u):=\big\{(x',0)\in B_1'\,:\, u(x',0)=0\big\},$$ 
and the \emph{free boundary $\Gamma_u$ of $u$} as the topological boundary of $\Delta(u)$ in $B_1'$. 
\smallskip

We say that $u$ has a \emph{unique blow-up limit at $x_0$}, if the sequence (the family) of functions
$$
\ds u_{x_0,r}:B_r\to\R,\qquad u_{x_0,r}(x)=\|u(r\cdot+x_0)\|_{L^2( \partial B_1)}^{-1}u(rx+x_0),
$$
converges weakly in $H^1(B_1, x_d^{1-2s} \mathcal L^d)$ to an admissible function $u_{x_0}$.
\smallskip

Here, building on the rectifiability of the free boundary, recently proved by Focardi and Spadaro (see Theorem~\ref{thm:FoSpa} below), and the classification of the two-dimensional homogeneous solutions, we prove that, at almost-every point of the free boundary, the blow-up is unique and corresponds to certain two-dimensional profiles with homogeneities $2m$, $2m-1+s$, or $2m+2s$. In particular, we answer a question left open in a recent paper of Focardi and Spadaro (see \cite{fosp2,fosp2-cor}). 
Our main result is the following.

\begin{teo}
\label{t:main}
Let $u$ be a solution of the lower dimensional obstacle problem \eqref{pbm:thin}. Then, for $\HH^{d-2}$-almost every $x_0\in \Gamma(u)$, the following does hold: 
\begin{enumerate}[(i)]
\item $u$ has a unique blow-up limit $u_{x_0}$ at $x_0$;
\item such blow up is either $2m$, $2m-1+s$, or  $2m+2s$ homogeneous, for some $m\in\N$; 
\item the blow-up limit $u_{x_0}:\R^d\to\R$ is of the form 
$$u_{x_0}(x',x_d)=\bar u (x' \cdot e, x_{d})\quad\text{for some vector}\quad e \in \mathbb S^{d-2}\subset \R^{d-1},$$
and $\bar u:\R^2 \to \R$ is a homogeneous solution of the lower dimensional obstacle problem \eqref{pbm:thin} in dimension two.
\end{enumerate}
\end{teo}

\begin{oss}[Lower dimensional obstacle problem {\small VS} minimal surfaces/harmonic maps]\label{r:minimal_surf} 
Our proof of Theorem \ref{t:main} is based on a very general dimension-reduction lemma (Lemma \ref{l:main}), which allows to reduce the question of the uniqueness of the blow-up limit to the analysis of the blow-up limits with a maximal number of symmetries. In fact, our argument is very general and can be applied in different contexts, for example, to the singular sets of minimal surfaces and harmonic maps. On the other hand, we notice that, in the case of the lower-dimensional (thin) obstacle problem, the blow-up limits with a maximal number of symmetries are completely described (for instance, in the case of the thin-obstacle problem, the homogeneous two-dimensional solutions are explicit), while for minimal surfaces and harmonic maps the singular blow-ups of minimal dimension (that is, with maximal number of symmetries) are not classified. 
However, combining the analogous version of Lemma \ref{l:main} for minimal surfaces and harmonic maps with the work of L. Simon \cite{Simon}, it is still possible to deduce uniqueness of the blow-up at almost every point of the singular set from its rectifiability (that is from Naber-Valtorta's result \cite{NaVa}). This is precisely the content of \cite[Remark 1.14]{Simon} and we will briefly explain it in Appendix \ref{a}.
\end{oss}


\section{Main lemma and proof of Theorem \ref{t:main}}

For every point $x_0\in B_1$, we define the {Almgren's frequency function}
$$
\ds N(u,x_0,r):=\frac{r\int_{B_r(x_0)}x_d^{1-2s} |\nabla u|^2\,dx}{\int_{\de B_r(x_0)}x_d^{1-2s} u^2\,d\HH^{d-1}}\,.
$$
The function $r\mapsto N(u,{x_0},r)$ is monotone nondecreasing in $r$ (see \cite{acs}), so that it is well defined the limit 
\begin{equation}\label{e:freq_in_zero}
N(u,x_0,0):=\lim_{r\to 0}N(u,{x_0},r).
\end{equation}
In particular, the free boundary can be decomposed according to the value of the frequency function at $r=0$. We denote the set of points of frequency $\lambda\in\R$ by  
$$\freq_{\lambda}(u):=\big\{x_0\in \Gamma(u)\,:\,N(u,x_0,0)=\lambda\big\}.$$
Our main lemma is the following.
%
\begin{lemma}[Splitting lemma]\label{l:main}
Let $u$ be a solution of the lower dimensional obstacle problem. Let $\lambda\in\R$ and $x_0\in \freq_{\lambda}(u)$ be a point of frequency $\lambda$ for which there exists a linear subspace $T_{x_0}$ of $\R^d$ satisfying the following property: 
\begin{itemize}
	\item [(SP)]For every $y_0\in T_{x_0}$ and sequence of radii $r_n$ converging to $0$, there is a sequence of points $y_n$ converging to $y_0$ such that $\displaystyle y_n\in \mathcal S_\lambda(u_{x_0,r_n})$, for every $n$.
\end{itemize}
Then, any blow-up limit $b$ of $u$ at $x_0$ is invariant in the direction of $T_{x_0}$, that is, 
\begin{equation}\label{e:invariance0}
b(x+y_0)=b(x)\quad\text{for every}\quad x\in\R^d\quad\text{and every}\quad y_0\in T_{x_0}.
\end{equation}
\end{lemma}	
\begin{oss}
We notice that in the proof of Lemma \ref{l:main}, we use only the following properties of the frequency function $N$:

\noindent$\bullet$ {\bf Monotonicity.} For every $x_0\in B_1$, the function $r\mapsto N(u,x_0,r)$ is non-decreasing.

\noindent$\bullet$ {\bf Scaling.} For $y_0\in B_1$, $s>0$ and $r>0$, such that $u_{x_0,r}$ is defined on the ball $B_s(y_0)$, we have
\begin{equation}\label{e:scaling}
N(u_{x_0,r},y_0,s)=N(u,x_0+ry_0,sr).
\end{equation}

\noindent$\bullet$ {\bf Continuity.} For every fixed $r>0$, the function $(u,x)\mapsto N(u,x,r)$, defined on $H^1(B_1)\times \R^d$ is continuous in the strong $H^1(B_1)\times \R^d$ topology. 

\noindent$\bullet$ {\bf Characterization of the homogeneous functions.}
Suppose that the point $x_0\in\R^d$ and the function $u:\R^d\to\R$ are such that 
$$N(u,x_0,r)=\lambda\quad\text{for every}\quad r>0.$$
Then $u$ is $\lambda$-homogeneous with respect to $x_0$, that is, 
$$u(x_0+rx)=r^\lambda u(x_0+x)\quad\text{for every}\quad x\in\R^d\quad\text{and}\quad r>0.$$
We also notice that the {\it monotonicity} property gives the existence of $N(u,x_0,0)$ (see \eqref{e:freq_in_zero}). Moreover, the {\it continuity} property implies the following:

\noindent$\bullet$ {\bf Upper semicontinuity.}
Suppose that $u_n:B_1\to\R$ is a sequence of functions converging strongly in $H^1(B _1)$ to a function $u_\infty\in H^1(B_1)$. Suppose that $x_n\in B_1$ be a sequence converging to some $x_\infty\in B_1$. Then
we have that 
\begin{equation}\label{e:usc}
N(u_\infty,x_\infty,0)\ge \limsup_{n\to\infty}N(u_n,x_n,0).
\end{equation}
Indeed, using the monotonicity of the function $r\mapsto N(u,x,r)$, we have 
$$N(u_\infty,x_\infty,r)=\lim_{n\to\infty}N(u_n,x_n,r)\ge \limsup_{n\to\infty}N(u_n,x_n,0).$$
Taking, the limit as $r\to0$, we get \eqref{e:usc}.
\end{oss}	
\begin{proof}[Proof of Lemma \ref{l:main}]
Let $b$ be any blow-up limit of $u$ at $x_0$. 
 Then, there is a sequence $r_n\to0$ such that $u_{r_n,x_0}$ converges to $b$ both strongly in $H^1_{\rm loc}$ and in $C^{1}_{\rm loc} ( \{x_d\geq 0\})$. 

We first claim that 
\begin{equation}\label{e:freq}
N(b,y_0,0)= \lambda\quad\text{for every}\quad y_0\in T_{x_0}.
\end{equation}
Indeed let $y_0\in T_{x_0}$ be fixed and let $ S_{\lambda}(u_{x_0,r_n})\ni y_n\to y_0$ be the sequence of points whose existence is guaranteed by (SP). In particular, since $u_{x_0,r_n}(y_n)=0$ and $u_{x_0,r_n}$ converges uniformly to $b$, we have that $b(y_0)=0$. 
\noindent  By the upper semi-continuity of $N$ we have that $N(b,y_0,0)\ge \lambda$. Indeed, since  $y_n\in \mathcal S_\lambda(u_{x_n,r_n})$ and $u_{x_0,r_n}$ converges to $b$ strongly in $H^1(B_1)$, we have 
$$N(b,y_0,0)\ge \limsup_{n\to\infty} N(u_{x_0,r_n},y_n,0)=\lambda.$$
On the other hand, $N(b,y_0,0)\le \lambda$. Indeed, by \eqref{e:scaling} and the fact that $b$ is homogeneous, we have that 
$$N(b,y_0,0)=\lim_{s\to0}N(b,y_0,s)=\lim_{s\to0}N(b_{0,r},y_0,s)=\lim_{s\to0}N(b,ry_0,rs)=N(b,ry_0,0),$$
for every $r>0$. In particular, this means that 
$$N(b,y_0,0)=\lim_{r\to0}N(b,ry_0,0)\le N(b,0,0)=\lambda,$$
where the inequality follows by the upper semi-continuity of the frequency function.
This concludes the proof of \eqref{e:freq}.

We next prove that the function $b$ is invariant in any direction $y \in T_{x_0}$, that is 
\begin{equation}\label{e:invariance}
b(x+ty)=b(x)\quad\text{for every}\quad x\in\R^d.
\end{equation}
Using the homogeneity of $b$ and \eqref{e:scaling}, for every $r>0$ we have that 
$$N(b,y,R)=N\Big(b_{0,R},\frac{y}{R},1\Big)=N\Big(b,\frac{y}{R},1\Big).$$
Taking the limit as $R\to\infty$, we get that 
$$
\lim_{R\to\infty}N(b,y,R)=\lim_{R\to\infty}N\left(b,\frac{y}{R},1\right)=N(b,0,1)=\lambda.
$$
In particular, together with \eqref{e:freq}, this implies that 
$$N(b,y,r)=\lambda\quad\text{for every}\quad r>0,$$
and so, $b$ is homogeneous with respect to $y$:
$$b(y+rx)=r^\lambda b(y+x)\quad\text{for every}\quad r>0.$$
Hence, for every  $x\in\R^d$ we can use the homogeneity with respect to $0$ and $y$ to obtain
$$b(x+y)=2^{\lambda}b\Big(\frac{x+y}{2}\Big)=2^{\lambda}b\Big(y+ \frac{x-y}{2}\Big)=b(x).$$
This concludes the proof of \eqref{e:invariance}. 
\end{proof}	

In the proof of Theorem \ref{t:main} we will use Lemma \ref{l:main} and the following recent result by Focardi and Spadaro, which we report here for the reader's convenience.

\begin{teo}[Focardi-Spadaro; see Theorem\,1.2 and Theorem\,1.3 of \cite{fosp2}]\label{thm:FoSpa}
	Let $u$ be a solution of the lower dimensional obstacle problem \eqref{pbm:thin} in $B_1$. Then $\Lambda(u)$ is a set of finite perimeter and there exists $\Sigma(u) \subseteq \Gamma(u)$ with Hausdorff dimension at most $n-2$ such that
	$$
	N(u,x_0,0) \in \{2m, 2m-1+s, 2m+2s\}_{m\in \N \setminus \{0\}} \qquad \text{for every}\qquad x_0\in \Gamma(u) \setminus \Sigma(u).
	$$
\end{teo}

\begin{proof}[\bf Proof of Theorem~\ref{t:main}]
Let 
$$\Sigma (u):=\Gamma(u)\setminus\left(\left(\bigcup_{m=1}^\infty \mathcal S_{2m}\right)\cup \left(\bigcup_{m=1}^\infty \mathcal S_{2m-1+s}\right)\cup \left(\bigcup_{m=1}^\infty \mathcal S_{2m+2s}\right)\right).$$
By \cite[Theorem 1.3]{fosp2}, we have that $\HH^{d-2}(\Sigma(u))=0$. Thus, it is sufficient to prove the claim for almost-every $x_0\in\mathcal S_{\lambda}$, where $\lambda=2m, 2m-1+s$ or $2m+2s$. Moreover, by \cite[Theorem 1.2]{fosp2}, we have that the free boundary $\Gamma(u)$ is $C^1$-rectifiable and so is each of the sets $\mathcal S_{2m-1+s}$, $\mathcal S_{2m}$ and $\mathcal S_{2m+2s}$ (for every $m\in\N$). In particular, this means that for almost every point $x_0$ of these sets, there exists a unique $(d-2)$-dimensional approximate tangent plane $T_{x_0} \subseteq \R^{d-1} \times \{0\}$, namely   
\begin{equation}
\label{eqn:approx-tang2}
\HH^{d-1} | \big(\mathcal S_{\lambda}(u_r)\cap B_1\big)\rightharpoonup \HH^{d-1} | ({T_{x_0} \cap B_1}) 
\end{equation} 
as locally finite measures.
 Hence, the splitting property hypothesis (SP) of Lemma \ref{l:main} is satisfied. Then Lemma \ref{l:main} implies that every blow-up limit $b$ of $u$ at $x_0$ is invariant with respect to a $(d-2)$-dimensional plane $T_{x_0}$. This means, that $b$ depends only on two variables: $x\cdot e$ and the last coordinate $x_d$, $e$ being (one of) the normal vector to $T_{x_0}$ in the hyperplane $\R^{d-1}$. Precisely, $b$ is of the form 
\begin{equation}
\label{form}
b(x)=\bar b(x\cdot e,x_d),
\end{equation}
where $\bar b$ is a homogeneous solution of the lower dimensional obstacle problem in dimension two.

We now consider the three cases $\lambda=2m$, $\lambda=2m-1+s$ and $\lambda=2m+2s$ separately. Indeed, we first notice that there is only one (up to a multiplicative constant) two-dimensional solution of the lower-dimensional obstacle problem of homogeneity $2m$. In particular, if $\lambda=2m$, then the  blow-up is unique and two-dimensional.

Let now $\lambda=2m-1+s$. In this case there are two two-dimensional homogeneous solutions (see for instance \cite{gape}) and so, two possible blow up limits of $u$ at $x_0$. We call them $b_1$ and  $b_2$. In order to prove the uniqueness of the blow-up as in statement (i) we have to exclude that, for two different sequences $r_j \to 0$ and $t_j \to 0$, the blow-up is $b_1$ and $b_2$, respectively. Indeed, taking the scalar product of $u_{x_0,r}$ with $b_1$ we see that
 $$\lim_{j\to \infty} \int_{\partial B_1} u_{x_0,r_j} b_1 =1\qquad\text{and}\qquad 
 \lim_{j\to \infty} \int_{\partial B_1} u_{x_0,t_j} b_1 =\int_{\partial B_1} b_1  b_2<1;$$ 
 hence, for every $j$, there exists $q_j\in (r_j, t_j)$ such that 
 $$
 \lim_{j\to \infty} \int_{\partial B_1} u_{x_0,q_j} b_1 =\frac 12 \Big(1+\int_{\partial B_1} b_1  b_2 \Big).
 $$
 This gives a contradiction. Indeed, up to a subsequence, $ u_{x_0,q_j} $ converges to a blow-up limit, which by Lemma \ref{l:main} should be $b_1$ or $b_2$. 



It now remains the case $\lambda=2m+2s$.  Fix $x_0\in\mathcal S_{2m+2s}(u)$ that admits a $(d-2)$-dimensional approximate tangent plane $T_{x_0}\subset \R^{d-1}\times \{0\}$ and such that, by \eqref{form}, every blow-up limit $b$ is  is of the form $b(x)=\bar b(x\cdot e,x_d),$ where $\bar b$ is a $(2m+2s)$-homogeneous solution in dimension two. It is sufficient to prove that $e$ is a normal vector to $T_{x_0}$. Let 
$$H_b:=\{x'\in\R^{d-1}\,:\, x'\cdot e=0\},$$
and suppose that there is a point $y_0\in T_{x_0}\setminus H_b\subset \R^{d-1}$. Without loss of generality, we may assume that $|y_0|=\sfrac12$. Let $u_{x_0,r_n}$ be a blow-up sequence converging to $b$. By definition of the tangent plane, there is a sequence of points $y_n\in B_1\cap \mathcal S_{2m+2s}(u_{r_n,x_0})$ such that $y_n\to y_0$. Since, $y_n$ are on the free boundary, there is a sequence of points $z_n$ in the non-contact set of $
u_{x_0,r_n}$ $\Big($that is, $u_{r_n,x_0}(z_n,0)>0$ and, as a consequence, $\frac{\partial u_{r_n,x_0}}{\partial x_d}(z_n,0)=0\Big)$ such that $z_n\to y_0$. 

When $s=\sfrac12$, we use the classification of the solutions in dimension two (see \cite{gape}), which implies that $\bar b:\R^2\to\R$ can be written (up to a positive multiplicative constant) in polar coordinates as
	\begin{equation}
	\label{eqn:blow-up-s-1/2}
	\bar b(r,\theta)=r^{2m+1}\sin\big(-(2m+1)\theta \big)\qquad\text{in}\qquad \{x_2\ge 0\},
	\end{equation}
	and it is reflected in an even way in the half-plane $\{x_2< 0\}$.
{In particular, $\ds \frac{\partial b}{\partial x_d}(y_0,0)<0$.
	On the other hand, the blow-up sequence $u_{x_0,r_n}$ converges in $C^1$ to the blow-up $b$ (see \cite{acs}). Thus, 
	$$\frac{\partial b}{\partial x_d}(y_0,0)=\lim_{n\to\infty}\frac{\partial u_{x_0,r_n}}{\partial x_d}(z_n,0)=0,$$
	which is a contradiction. In conclusion, $T_{x_0}=H_b$, so the vector $e$ and the blow-up $b$ are uniquely determined by the tangent plane $T_{x_0}$. This concludes the analysis of $\mathcal S_{2m+2s}$ in the case $s=\sfrac12$.
}
		
For general $s$, a nice formula as \eqref{eqn:blow-up-s-1/2} is not available but the two-dimensional solutions are described in detail in \cite[Appendix A.1]{fosp2} and the uniqueness of the blow-up follows by a similar argument. Indeed, by \cite[equation (A.4)]{fosp2}, up to a multiplicative constant, we have that 
$$\bar b(x_1,x_2) = |x_2|^{2s}\big(-1+ O(|(x- y_0|^2)\big).$$ 
This means that, at the point $y_0$, we have 
$$|x_d|^{1-2s}\frac{\partial b}{\partial x_d}(y_0,0):=\lim_{x_d\to 0}|x_d|^{1-2s}\frac{\partial b(y_0,x_d)}{\partial x_d}<0.$$
On the other hand, by \cite[Theorem 2.1]{fosp2} we have that 
$$|x_d|^{1-2s}\frac{\partial b}{\partial x_d}(y_0,0)=\lim_{n\to\infty}|x_d|^{1-2s}\frac{\partial u_{x_0,r_n}}{\partial x_d}(z_n,0)=0,$$
where the last inequality is due to the fact that $z_n$ is not on the contact set (see for instance  \cite[Corollary 2.4]{fosp2}).
This is a contradiction. Thus, also in the case $s\neq\sfrac12$ and $\lambda=2m+2s$, the blow-up is unique (as it is uniquely determined by $T_{x_0}$). This concludes the proof. 
		\end{proof}
		
		\appendix

	\section{About Remark {\ref{r:minimal_surf}}}\label{a}
	 In this section we elaborate a bit more on Remark \ref{r:minimal_surf} in the particular case of minimal surfaces (although the same holds for harmonic maps). Following the notations of \cite{Simon}, we denote with $\mathcal M$ a multiplicity one class of $n$-dimensional minimal surfaces and we denote with $\sing\,M$, the singular set of $M\in \mathcal M$. Moreover we let
	$$
	m:=\max \{\dim \sing\,M\,:\, M\in \mathcal M\}\,.
	$$  	
	Thanks to a result of Naber-Valtorta \cite{NaVa}, we know that $\sing\,M$ has finite $\HH^m$-volume and it is locally $\HH^m$-rectifiable. Next, let us denote with $\Theta_M(x)$ the density of $M\in \mathcal M$ at a point $x$, and recall that a consequence of \L ojasiewicz inequality for minimal surfaces is that the set of admissible densities is discrete, that is, 
	$$\big\{ \Theta_\bC(0)\,:\, \bC \mbox{ stationary cone with }\dim (\sing\,\bC)=m \big\}=\{\alpha_1,\dots,\alpha_N\},$$ 
	with $\alpha_1<\dots<\alpha_N$ (see \cite[4.3 Lemma]{Simon}). Consider the sets
	$$
	\mathcal S_{j}:=\big\{ x\in \sing\,M\,:\, \Theta_M(x)=\alpha_j \big\}\,\qquad j=1,\dots,N\,,
	$$
	and notice that, by standard stratification arguments, 
	\begin{equation}\label{e:ms1}
	\HH^m\left(\sing\,M\setminus \Big(\bigcup_{j=1}^N \mathcal S_j\Big)\right)=0\,.
	\end{equation}
	As a consequence (of the analogous) of Lemma \ref{l:main}, applied to this case, we know that
	\begin{itemize}
		\item[(MS)] for every point $x\in \mathcal S_j$ for which the approximate tangent space $T_x$ to $\mathcal S_j$ at $x$ exists, all the tangent cones $\bC$ to $M$ at $x$ are such that $\dim(\sing\,\bC)=m$ and moreover $T_x\subset \bC$.
	\end{itemize}
	Thanks to the Naber-Valtorta rectifiability result, this is the case for $\HH^m$-a.e. point of $\mathcal S_j$, that is
	\begin{itemize}
	\item[(MS')] for $\HH^m$-a.e. $x\in \mathcal S_j$, all the tangent cones $\bC$ to $M$ at $x$ are such that 
	$$\dim(\sing\,\bC)=m\qquad\text{and}\qquad T_x\subset \sing\,\bC.$$
	\end{itemize}
	It follows from (MS') and \eqref{e:ms1}, combined with standard arguments that, for $\HH^m$-a.e. $x\in \sing\,M$, there is an $m$-dimensional subspace $L_x$ such that, for every $\eps>0$,
	\begin{gather}
	B_1(0)\cap \eta_{x,\sigma}(\sing\,M)\subset \mbox{the $\eps$-neighborhood of $L_x$} \label{e:ms2}\\
	B_1(0)\cap L_x\subset \mbox{the $\eps$-neighborhood of $\eta_{x,\sigma}(\sing\,M)$} \label{e:ms3}\,,
	\end{gather}
	where $\eta_{x,\sigma}(y):=\sigma^{-1} (y-x)$. Indeed, if $L_x=T_x$ is as in (MS'), then \eqref{e:ms3} follows immediately by the definition of approximate tangent, while \eqref{e:ms2} follows from (MS'), the upper semicontinuity of the density and a simple blow-up argument.  
	
	Now, the main content of \cite{Simon} is precisely to show that at $\HH^m$-a.e. $x\in \sing\,M$, for which \eqref{e:ms2} and \eqref{e:ms3} do hold, the blow-up is unique (see the second part of \cite[Proof of Remark 1.14]{Simon}). Indeed, these are the points where no $\delta$-gap nor $\delta$-tilt happens.
	\smallskip

	Finally, we notice that, for the thin obstacle problem and the minimal surfaces, the set of points at which the blow-up limit is unique is characterized differently. In the case if the lower-dimensional (thin) obstacle problem, the blow-up is unique at every point at which the free boundary admits an approximate tangent plane. On the other hand, for minimal surfaces, the blow-up is unique at almost-every point satisfying the conditions \eqref{e:ms2} and \eqref{e:ms3} (this is due to the fact that the uniqueness is achieved by an averaging process), which (as we noticed above) turn out to be fulfilled whenever the singular set admits an approximate tangent plane. In particular, for minimal surfaces we cannot characterize the points with unique blow-up as the ones at which the approximate tangent plane to $\mathcal S_j$ exists. However, this would be the case if we knew a priori that the $(n-m)$-dimensional minimal cones are integrable. Precisely, if the $(n-m)$-dimensional minimal cones were integrable, then the blow-up would be unique at every point satisfying \eqref{e:ms2} and \eqref{e:ms3} (see for instance \cite{Simon2}).  
	
	

\bigskip\bigskip
\noindent {\bf Acknowledgments.} 
The first author acknowledges the support of Dr.\,Max R\"ossler, of the Walter Haefner Foundation, the ETH Z\"urich Foundation and the SNF grant 182565. The second author has been partially supported by the NSF grant DMS 1810645. Most of this paper was written during a meeting in {\it MFO - Oberwolfach Research Institute for Mathematics}; the three authors are grateful to MFO for the kind hospitality. 

Finally, we also wish to acknowledge Xavier Ros-Oton for pointing out a mistake in the first version of the preprint.

\end{document}